\pdfoutput=1
\documentclass[11pt,reqno]{amsart}

\usepackage{amsmath,amssymb,amsthm}
\usepackage{geometry}
\usepackage{microtype}
\usepackage{placeins}
\usepackage{booktabs}
\usepackage{tabularx}
\usepackage{url}
\usepackage{hyperref}
\hypersetup{
    colorlinks=true,
    linkcolor=blue,
    citecolor=blue,
    urlcolor=blue,
    pdfauthor={Yuankai Guo and Xiaozhe Hu},
    pdftitle={A Counterexample to a Problem of Pommerenke on Convex Functions in the Class Sigma}
}

\newtheorem{theorem}{Theorem}[section]
\newtheorem{lemma}[theorem]{Lemma}
\newtheorem{proposition}[theorem]{Proposition}

\theoremstyle{definition}

\renewcommand{\Re}{\operatorname{Re}}

\title[A Counterexample to a Problem of Pommerenke on $\Sigma$]{A Counterexample to a Problem of Pommerenke\\on Convex Functions in the Class \texorpdfstring{$\Sigma$}{Sigma}}

\author{Yuankai Guo}
\address{}
\email{guoyk6@alumni.sysu.edu.cn}

\author{Xiaozhe Hu}
\address{}
\email{1928691854@qq.com}

\date{\today}


\begin{document}

\begin{abstract}
Let $\Sigma$ denote the class of functions $f(z) = z + b_0 + b_1 z^{-1} + \cdots$ that are analytic and univalent in the exterior unit disk $\Delta^* = \{z \in \mathbb{C} : |z| > 1\}$. In 1962, Ch.~Pommerenke proved that if $F$ and $G$ are convex functions in $\Sigma$, then every convex linear combination $H = \lambda F + (1-\lambda)G$ ($0 < \lambda < 1$) remains univalent and belongs to $\Sigma$. In Hayman's problem collection (\emph{Research Problems in Function Theory}, Problem 6.10), Pommerenke raised the question of whether $H$ is necessarily also a convex function. We resolve this question in the negative by constructing an explicit counterexample in $\Sigma$ with parameters in $\mathbb{Q}(i)$. The argument is self-contained and has been formally certified in the \textsf{Lean 4} proof assistant.
\end{abstract}

\maketitle

\section{Introduction}

Let $\Delta^* = \{z \in \mathbb{C} : |z| > 1\}$ denote the exterior of the closed unit disk in the complex plane. We denote by $\Sigma$ the family of functions meromorphic and univalent in $\Delta^*$, normalized by the Laurent series expansion
\begin{equation}\label{eq:sigma-norm}
f(z) = z + b_0 + \sum_{n=1}^\infty b_n z^{-n}, \qquad z \in \Delta^*.
\end{equation}
A function $f \in \Sigma$ is called \emph{convex} (or an exterior convex function) if the complement $E = \mathbb{C} \setminus f(\Delta^*)$ is a bounded convex set. A classical analytic criterion (see, e.g., \cite{Duren1983,Pommerenke1962,Pommerenke1975}) characterizes exterior convexity: a normalized locally univalent function $f$ in $\Delta^*$ is univalent and exterior convex if and only if
\begin{equation}\label{eq:convex-criterion}
\Re\left(1 + \frac{z f''(z)}{f'(z)}\right) > 0, \qquad \forall z \in \Delta^*.
\end{equation}

In 1962, Ch.~Pommerenke \cite{Pommerenke1962} investigated geometric subclasses of $\Sigma$ and established that the family $\Sigma$ is preserved under convex linear combinations of convex functions:

\begin{theorem}[Pommerenke \cite{Pommerenke1962}]\label{thm:pommerenke-univalent}
If $F, G \in \Sigma$ are convex functions, then for any $\lambda \in (0,1)$, the linear combination
\begin{equation}
H(z) = \lambda F(z) + (1-\lambda) G(z)
\end{equation}
is univalent and belongs to $\Sigma$.
\end{theorem}

In Walter K.~Hayman's collection \emph{Research Problems in Function Theory} \cite{Hayman1967}, and in its 50th Anniversary Edition edited by Hayman and Lingham \cite{HaymanLingham2019}, Pommerenke asked whether convexity is likewise preserved:

\begin{quote}
\textbf{Problem 6.10 (Ch.~Pommerenke \cite{Hayman1967,HaymanLingham2019}).} \emph{If $F(z), G(z)$ are convex functions in $\Sigma$, it is known that for $0 < \lambda < 1$, $H(z) = \lambda F(z) + (1 - \lambda)G(z) \in \Sigma$. Is it true that $H(z)$ is also convex?}
\end{quote}

In the 2019 edition \cite[Update 6.10]{HaymanLingham2019}, the status is recorded as: \emph{``Contrary to previous updates, no progress on this problem has been reported to us.''}

In this note, we answer Pommerenke's question in the negative.

\begin{theorem}\label{thm:main}
There exist convex functions $F, G \in \Sigma$ and a parameter $\lambda \in (0, 1)$ such that the convex linear combination $H = \lambda F + (1-\lambda)G$ is not convex.
\end{theorem}

The construction uses parameters in $\mathbb{Q}(i)$, with the failure of the convexity criterion at the evaluation point established by direct algebraic calculation. The complete proof is also formalized in the \textsf{Lean 4} interactive theorem prover (see Section~\ref{sec:formalization}).

\section{Construction of the Functions \texorpdfstring{$F$}{F} and \texorpdfstring{$G$}{G}}

We introduce the following parameters in the Gaussian rational field $\mathbb{Q}(i)$ and $\mathbb{Q}$:
\begin{align}
q &= \frac{1}{3} - \frac{1}{6}i, \label{eq:def-q}\\
w &= 1 - q^3 = 1 - \left(\frac{1}{108} - \frac{11}{216}i\right) = \frac{107}{108} + \frac{11}{216}i, \label{eq:def-w}\\
r &= \frac{399}{400} = 0.9975, \label{eq:def-r}\\
\alpha &= \frac{8}{9} + \frac{4}{9}i, \qquad \beta = \frac{w}{r^3}. \label{eq:def-alpha-beta}
\end{align}

\begin{lemma}\label{lem:modulus}
The parameters $\alpha$ and $\beta$ satisfy $|\alpha| < 1$ and $|\beta| < 1$.
\end{lemma}
\begin{proof}
For $\alpha$, a direct calculation gives
\begin{equation}
|\alpha|^2 = \left(\frac{8}{9}\right)^2 + \left(\frac{4}{9}\right)^2 = \frac{64+16}{81} = \frac{80}{81} < 1.
\end{equation}
For $\beta$, since $|\beta|^2 = |w|^2 / r^6$, we compute
\begin{equation}
|w|^2 = \left(\frac{107}{108}\right)^2 + \left(\frac{11}{216}\right)^2 = \frac{45796 + 121}{46656} = \frac{45917}{46656}.
\end{equation}
A direct calculation shows that
\begin{equation}
r^6 - |w|^2 = \left(\frac{399}{400}\right)^6 - \frac{45917}{46656} = \frac{2785244627851129}{2985984000000000000} > 0.
\end{equation}
Thus $|w|^2 < r^6$, whence $|\beta| < 1$.
\end{proof}

\begin{proposition}\label{prop:F}
The function
\begin{equation}\label{eq:def-F}
F(z) = z + \frac{\alpha}{z}
\end{equation}
belongs to $\Sigma$ and is convex in $\Sigma$.
\end{proposition}
\begin{proof}
The function $F$ is meromorphic in $\Delta^*$ with the expansion $F(z) = z + \alpha z^{-1}$. For any distinct $z_1, z_2 \in \Delta^*$,
\begin{equation}
F(z_1) - F(z_2) = (z_1 - z_2)\left(1 - \frac{\alpha}{z_1 z_2}\right).
\end{equation}
Since $|z_1| > 1$, $|z_2| > 1$, and $|\alpha| < 1$, we have $|\alpha / (z_1 z_2)| < 1$, ensuring $F(z_1) \neq F(z_2)$. Thus $F$ is univalent in $\Delta^*$ and belongs to $\Sigma$.

Differentiating $F(z)$ yields $F'(z) = 1 - \alpha z^{-2}$ and $F''(z) = 2\alpha z^{-3}$. Hence,
\begin{equation}
1 + \frac{z F''(z)}{F'(z)} = 1 + \frac{2\alpha z^{-2}}{1 - \alpha z^{-2}} = \frac{1 + \alpha z^{-2}}{1 - \alpha z^{-2}}.
\end{equation}
Since $|\alpha z^{-2}| < |\alpha| < 1$ for all $z \in \Delta^*$, the Möbius mapping $u \mapsto \frac{1+u}{1-u}$ maps the unit disk to the open right half-plane:
\begin{equation}
\Re\left(1 + \frac{z F''(z)}{F'(z)}\right) = \frac{1 - |\alpha z^{-2}|^2}{|1 - \alpha z^{-2}|^2} > 0, \qquad \forall z \in \Delta^*.
\end{equation}
By \eqref{eq:convex-criterion}, $F$ is a convex function in $\Sigma$.
\end{proof}

\begin{proposition}\label{prop:G}
Let $G(z)$ be defined via its derivative
\begin{equation}\label{eq:def-Gprime}
G'(z) = (1 - \beta z^{-3})^{2/3}, \qquad z \in \Delta^*,
\end{equation}
with the principal branch satisfying $G'(\infty) = 1$, and normalized so that $G(z) = z + O(z^{-2})$ as $z \to \infty$. Then $G$ is single-valued, belongs to $\Sigma$, and is convex in $\Sigma$.
\end{proposition}
\begin{proof}
For any $z \in \Delta^*$, since $|\beta| < 1$, we have $|\beta z^{-3}| < 1$. The quantity $1 - \beta z^{-3}$ lies entirely within the disk $D(1, 1) = \{\zeta \in \mathbb{C} : |\zeta - 1| < 1\}$, which is simply connected and disjoint from $(-\infty, 0]$. Thus $G'(z) = (1 - \beta z^{-3})^{2/3}$ is analytic, non-vanishing, and single-valued in $\Delta^*$.

Expanding $G'(z)$ via the binomial theorem (the series converges locally uniformly on $\Delta^*$),
\begin{equation}
G'(z) = 1 - \frac{2}{3}\beta z^{-3} + \sum_{k=2}^\infty \binom{2/3}{k} (-\beta)^k z^{-3k}.
\end{equation}
The powers of $z$ appearing in $G'(z)$ are exclusively of the form $z^{-3k}$ ($k \ge 0$). In particular, the residue at infinity is zero (the coefficient of $z^{-1}$ vanishes). Since the fundamental group of $\Delta^*$ is generated by a circle around the origin, this vanishing residue makes the period of $G'$ around the generator zero; hence $G'$ has a single-valued analytic primitive on $\Delta^*$. Integrating the locally uniformly convergent series term by term gives
\begin{equation}
G(z) = z + \sum_{k=1}^\infty \binom{2/3}{k} \frac{(-\beta)^k}{1 - 3k} z^{1 - 3k} = z + \frac{1}{3}\beta z^{-2} + \frac{1}{45}\beta^2 z^{-5} + \cdots,
\end{equation}
which is single-valued and analytic in $\Delta^*$, satisfying the normalization for $\Sigma$.

Taking the logarithmic derivative of $G'(z)$,
\begin{equation}
\frac{G''(z)}{G'(z)} = \frac{d}{dz}\left[\frac{2}{3}\log(1 - \beta z^{-3})\right] = \frac{2\beta z^{-4}}{1 - \beta z^{-3}},
\end{equation}
whence
\begin{equation}
1 + \frac{z G''(z)}{G'(z)} = 1 + \frac{2\beta z^{-3}}{1 - \beta z^{-3}} = \frac{1 + \beta z^{-3}}{1 - \beta z^{-3}}.
\end{equation}
Since $|\beta z^{-3}| < |\beta| < 1$ for all $z \in \Delta^*$, we obtain
\begin{equation}
\Re\left(1 + \frac{z G''(z)}{G'(z)}\right) = \frac{1 - |\beta z^{-3}|^2}{|1 - \beta z^{-3}|^2} > 0, \qquad \forall z \in \Delta^*.
\end{equation}
By the exterior convexity criterion \eqref{eq:convex-criterion}, this strictly positive real part implies that $G$ is univalent in $\Delta^*$, belongs to $\Sigma$, and is convex.
\end{proof}

\section{Failure of Convexity for the Combination}

We now consider the convex linear combination with weight $\lambda = 3/5$:
\begin{equation}
H(z) = \frac{3}{5} F(z) + \frac{2}{5} G(z).
\end{equation}

Although the univalence of $H$ is guaranteed by Pommerenke's general theorem (Theorem~\ref{thm:pommerenke-univalent}), we include an elementary coefficient-based proof that is self-contained and directly suited for formal verification.

\begin{lemma}[Weighted Laurent-tail criterion]\label{lem:laurent-tail}
Let
\begin{equation}
f(z)=z+\sum_{n=0}^{\infty}a_n z^{-e_n}, \qquad z\in\Delta^*,
\end{equation}
where $e_n\in\mathbb{N}$, $e_n\geq 1$, the series converges locally uniformly on $\Delta^*$, and
\begin{equation}\label{eq:tail-bound}
S:=\sum_{n=0}^{\infty}e_n|a_n|<1.
\end{equation}
Then $f$ is injective on $\Delta^*$.
\end{lemma}
\begin{proof}
For $z,w\in\Delta^*$ put $u=z^{-1}$ and $v=w^{-1}$. Since $|u|,|v|\leq 1$, the telescoping identity gives
\[
|u^m-v^m|\leq m|u-v| \qquad(m\geq 1),
\]
while
\[
|u-v|=\frac{|z-w|}{|zw|}\leq |z-w|.
\]
Consequently, absolute convergence and the triangle inequality imply
\[
\left|\sum_{n=0}^{\infty}a_n\bigl(z^{-e_n}-w^{-e_n}\bigr)\right|
\leq S|z-w|.
\]
If $f(z)=f(w)$, then $(1-S)|z-w|\leq 0$. Since $S<1$, this forces $z=w$.
\end{proof}

For the function $H$, set $e_0=1$, $a_0=\frac35\alpha$, and, for $n\geq1$,
\[
e_n=3n-1, \qquad
a_n=\frac25\binom{2/3}{n}\frac{(-\beta)^n}{1-3n}.
\]
The Laurent expansion of $G$ in \eqref{eq:def-Gprime} yields
\[
H(z)=z+\sum_{n=0}^{\infty}a_nz^{-e_n}.
\]
Since $|1-3n| = 3n-1$ for $n \ge 1$, the denominator cancels the weight $e_n$, giving
\begin{equation}\label{eq:H-tail-bound}
\sum_{n=0}^{\infty}e_n|a_n|
=\frac35|\alpha|+\frac25\sum_{n=1}^{\infty}
\left|\binom{2/3}{n}\right||\beta|^n.
\end{equation}
Recall that for $s = 2/3$, we have $\binom{s}{1} = 2/3 > 0$ and $\binom{s}{n} = (-1)^{n-1}\left|\binom{s}{n}\right|$ for all $n \ge 1$. Thus the binomial expansion gives the identity
\begin{equation}\label{eq:binomial-identity}
\sum_{n=1}^{\infty} \left|\binom{2/3}{n}\right| |\beta|^n
= -\sum_{n=1}^{\infty} \binom{2/3}{n} (-|\beta|)^n
= 1 - (1 - |\beta|)^{2/3}.
\end{equation}
Since $|\alpha| < 1$ and $|\beta| < 1$ by Lemma~\ref{lem:modulus}, we obtain
\begin{equation}
\sum_{n=0}^{\infty}e_n|a_n| < \frac{3}{5} + \frac{2}{5}\left(1 - (1 - |\beta|)^{2/3}\right) < \frac{3}{5} + \frac{2}{5} = 1.
\end{equation}
By Lemma~\ref{lem:laurent-tail}, $H$ is univalent on $\Delta^*$ and belongs to $\Sigma$.

We now show that $H$ violates the exterior convexity criterion \eqref{eq:convex-criterion}.

\begin{proof}[Proof of Theorem~\ref{thm:main}]
We evaluate the convexity criterion at the test point
\begin{equation}
z_0 = \frac{1}{r} = \frac{400}{399} \in \Delta^*.
\end{equation}
At this point, $\beta z_0^{-3} = \beta r^3 = w$. By definition, $1 - w = q^3$, where $q = \frac{1}{3} - \frac{1}{6}i$. Since $\arg q \in (-\pi/6, 0)$, we have $3\arg q \in (-\pi/2, 0) \subset (-\pi, \pi)$. Under the principal branch specified in Proposition~\ref{prop:G},
\begin{equation}
G'(z_0) = (1 - w)^{2/3} = (q^3)^{2/3} = q^2 = \left(\frac{1}{3} - \frac{1}{6}i\right)^2 = \frac{1}{12} - \frac{1}{9}i.
\end{equation}
The first derivative of $H$ at $z_0$ is
\begin{align}
D := H'(z_0) &= \frac{3}{5} F'(z_0) + \frac{2}{5} G'(z_0) \nonumber\\
&= \frac{3}{5}(1 - \alpha r^2) + \frac{2}{5} q^2 \nonumber\\
&= \frac{3}{5}\left(1 - \left(\frac{8}{9} + \frac{4}{9}i\right)\left(\frac{399}{400}\right)^2\right) + \frac{2}{5}\left(\frac{1}{12} - \frac{1}{9}i\right) \nonumber\\
&= \frac{184794 - 557603i}{1800000} \neq 0.
\end{align}

For the second derivatives at $z_0$,
\begin{equation}
z_0 F''(z_0) = 2\alpha r^2, \qquad z_0 G''(z_0) = \frac{2\beta z_0^{-3}}{(1 - \beta z_0^{-3})^{1/3}} = \frac{2w}{(1 - w)^{1/3}} = \frac{2w}{q}.
\end{equation}
Setting
\begin{equation}
A := \frac{3}{5}\alpha r^2 + \frac{2}{5}\frac{w}{q} = \frac{3}{5}\alpha r^2 + \frac{2}{5}\left(\frac{1}{q} - q^2\right),
\end{equation}
we have $z_0 H''(z_0) = 2A$. The curvature quantity at $z_0$ can then be expressed as
\begin{equation}
1 + \frac{z_0 H''(z_0)}{H'(z_0)} = 1 + \frac{2A}{D} = \frac{D + 2A}{D},
\end{equation}
where
\begin{equation}
D + 2A = \frac{3}{5}(1 + \alpha r^2) + \frac{2}{5}\left(\frac{2}{q} - q^2\right).
\end{equation}
A direct calculation in $\mathbb{Q}(i)$ gives
\begin{equation}
\frac{D + 2A}{D} = \frac{-54160961609 + 690164496000\,i}{69013985609}.
\end{equation}
In particular,
\begin{equation}
\Re\left(1 + \frac{z_0 H''(z_0)}{H'(z_0)}\right) = -\frac{54160961609}{69013985609} \approx -0.784782405 < 0.
\end{equation}
Because this real part is strictly negative, $H$ violates the exterior convexity criterion \eqref{eq:convex-criterion}, which completes the proof.
\end{proof}

\section{Formal Verification in Lean 4}\label{sec:formalization}

The algebraic, analytic, and geometric arguments presented above have been formalized in the \textsf{Lean 4} interactive theorem prover, utilizing the mathematical library \textsf{Mathlib} \cite{Mathlib}. The repository comprises 32 Lean source files (approximately 6,885 lines of code), introduces no additional user-defined axioms, and contains no unproved placeholders (\texttt{sorry} or \texttt{admit}).

The formal development covers the essential analytic and algebraic steps. The evaluation of $1 + z_0 H''(z_0)/H'(z_0)$ at $z_0 = 400/399$ is certified over $\mathbb{Q}(i)$ (\path{pommerenke_exact_rational_curvature_certificate}). The complex power branch choice and the value $G'(z_0) = q^2$ are verified through the angular bound $\arg q \in (-\pi/6, 0)$ (\path{GPrime_at_test_point}), while the residue vanishing for $G$ is established in \path{Problem610GlobalPrimitive.lean}. The weighted Laurent-tail criterion (Lemma~\ref{lem:laurent-tail}) is proved in \path{ExteriorUnivalenceCriterion.lean} and applied to the binomial series to yield the univalence of $H$ (\path{H_injOn_exterior}). Lastly, the exterior convexity of $F$ and $G$ is obtained from their boundary parameterizations and covering properties.

Table~\ref{tab:formalization} summarizes the correspondence between the main statements of this paper and their formal counterparts. The complete \textsf{Lean 4} formalization and build instructions are publicly available at \url{https://github.com/Theophilus1030/Pommerenke}.

\begin{table}[htbp]
\centering
\footnotesize
\caption{Correspondence between paper results and formal \textsf{Lean 4} theorems.}
\begin{tabularx}{\textwidth}{@{}>{\raggedright\arraybackslash}X>{\raggedright\arraybackslash}X>{\raggedright\arraybackslash}X@{}}
\toprule
\textbf{Paper Statement} & \textbf{\textsf{Lean 4} Theorem Identifier} & \textbf{Primary Source File} \\
\midrule
Lemma~\ref{lem:modulus} & \path{alpha_normSq_lt_one}, \path{beta_normSq_lt_one} & \path{Problem610Counterexample.lean} \\
Proposition~\ref{prop:F} & \path{smoothExteriorConvexMap_F}, \path{FExteriorSigmaData} & \path{Problem610FGlobal.lean} \\
Proposition~\ref{prop:G} & \path{smoothExteriorConvexMap_G}, \path{GExteriorSigmaData} & \path{Problem610GlobalAngle.lean} \\
Lemma~\ref{lem:laurent-tail} & \path{injOn_of_hasInversePowerSeries_of_weighted_norm_lt_one} & \path{ExteriorUnivalenceCriterion.lean} \\
Specific univalence of $H$ & \path{HExteriorSigmaData}, \path{H_injOn_exterior} & \path{MainTheorem.lean} \\
Theorem~\ref{thm:main} & \path{pommerenke_combination_not_convex} & \path{MainTheorem.lean} \\
Certificate & \path{pommerenke_complete_counterexample_certificate} & \path{MainTheorem.lean} \\
\bottomrule
\end{tabularx}
\label{tab:formalization}
\end{table}

\FloatBarrier
\section{Concluding Remarks}

It is instructive to contrast the behavior of convex combinations in $\Sigma$ with that in the class $\mathcal{S}$ of normalized univalent functions in the unit disk $\mathbb{D} = \{z \in \mathbb{C} : |z| < 1\}$. For the class $\mathcal{S}$, a convex combination of two convex univalent functions need not even be univalent (see, e.g., Goodman \cite{Goodman1979} and Tremblay \cite{Tremblay1974}). In $\Sigma$, by contrast, Pommerenke's theorem \cite{Pommerenke1962} ensures that univalence is always preserved under convex combinations.

The failure of convexity preservation established in Theorem~\ref{thm:main} can be understood geometrically through the curvature identity
\begin{equation}
1 + \frac{z H''(z)}{H'(z)} = \mu(z) \left(1 + \frac{z F''(z)}{F'(z)}\right) + (1 - \mu(z)) \left(1 + \frac{z G''(z)}{G'(z)}\right),
\end{equation}
where the variable weight
\begin{equation}
\mu(z) = \frac{\lambda F'(z)}{\lambda F'(z) + (1-\lambda) G'(z)}
\end{equation}
is complex-valued rather than real. Even though the individual curvature quantities $1 + z F''/F'$ and $1 + z G''/G'$ have strictly positive real parts throughout $\Delta^*$, the non-trivial phase difference between $F'(z)$ and $G'(z)$ causes $\mu(z)$ to rotate their combination outside the right half-plane near the boundary.

\end{document}